\documentclass[12pt]{amsart}
\usepackage{amssymb}
\usepackage{graphicx}
\usepackage{colordvi}
\usepackage{youngtab}
\usepackage{tikz}
\usepackage{apacite}

\numberwithin{equation}{section}
\newtheorem{theorem}{Theorem}[section]
\newtheorem{proposition}[theorem]{Proposition}
\newtheorem{lemma}[theorem]{Lemma}
\newtheorem{corollary}[theorem]{Corollary}
\theoremstyle{remark}
\newtheorem{example}[theorem]{Example}
\newtheorem{remark}[theorem]{Remark}

\newcounter{FNC}[page]
\def\fauxfootnote#1{{\addtocounter{FNC}{2}$^\fnsymbol{FNC}$%
     \let\thefootnote\relax\footnotetext{$^\fnsymbol{FNC}$#1}}}
\newcommand{\C}{{\mathbb{C}}}
\newcommand{\N}{{\mathbb{N}}}
\newcommand{\Z}{{\mathbb{Z}}}
\newcommand{\R}{{\mathbb{R}}}
\newcommand{\Q}{{\mathbb{Q}}}

\title{ Grassmannians in the Lattice points of Dilations of the Standard Simplex}

\author{Praise Adeyemo}
\address{Department of Mathematics\\
  University of Ibadan\\
   Ibadan, Oyo, Nigeria}
\email{ph.adeyemo@ui.edu.ng, praise.adeyemo13@gmail.com}
\urladdr{http://sci.ui.edu.ng/HPAdeyemo}

\subjclass{14M15, 14N15, 05E05}
\keywords{Lattice polytope, Grassmannian and Partition}

\copyrightinfo{}{}

\begin{document}

\begin{abstract}
 A remarkable connection between the cohomology ring ${\rm H^{\ast}(Gr}(d, d+r),\Z)$ of the Grasssmannian ${\rm Gr}(d,d+r)$ and the lattice points of the dilation $r\Delta_{d}$ of the standard d-simplex is investigated. The natural grading on the cohomology induces  different gradings  of the lattice points of $r\Delta_{d}$. This leads to different refinements of the Ehrhart polynomial $L_{\Delta_{d}}(r)$ of the standard $d$-simplex. We study two of these refinements which are defined by the weights $(1,1,\dots,1)$ and $(1,2,\dots, d)$. One of the refinements interprets the Poincar\'e 
polynomial ${\rm P(Gr}(d,d+r),z)$  as the counting of the lattice points which lie on the slicing hyperplanes of the dilation  $r\Delta_d$.  Therefore, on the combinatorial level the Poincar\'e polynomial of the Grassmannian Gr$(d,d+r)$ is a refinement of the Ehrhart polynomial $L_{\Delta_d}(r)$ of the standard $d$-simplex $\Delta_{d}$.

\end{abstract}

\maketitle
%
\section{Introduction}
\noindent Consider the diagonal sequence $D_d$ of natural numbers realized from the Pascal triangle illustrated below\\
$$\begin{tabular}{cccccccccccccccccc} 
 &  &  &  &  &  &  \rotatebox{320}{$D_1:$} &  & 1 &  & &  &  &  &  &  & &  \\ 
 &  &  &  &  &\rotatebox{320}{$D_2:$}  &  & 1 &  & 1 &  &  &  &  &  &  &  & \\ 
 &  &  &  &\rotatebox{320}{$D_3:$}  &  & 1 &  & 2 &  & 1 &  &  &  &  &  &  &\\ 
 &  &  &\rotatebox{320}{$D_4:$}  &  & 1 &  & 3 &  & 3 &  & 1 &  &  &  &  &  &\\ 
 &  &\rotatebox{320}{$D_5:$}  &  & 1 &  & 4 &  &  6 &  & 4 &  & 1 &  &  &  &  &\\ 
 &\rotatebox{320}{$D_6:$}  &  & 1 &  & 5 &  & 10 &  & 10 &  & 5 &  & 1 &  &  & & \\ 
 \rotatebox{320}{$\vdots$} &  & 1 &  & 6 &  & 15 &  & 20 &  & 15 &  & 6 &  & 1 &   & &  \\ 
 &  & \ \ \ \rotatebox{45}{$\vdots$}  &   & \ \ \ \ \rotatebox{45}{$\vdots$}  &  &\ \ \ \ \rotatebox{45}{$\vdots$} &   &  &  \rotatebox{45}{$\vdots$} &  &  \rotatebox{45}{$\vdots$} &  &  \rotatebox{45}{$\vdots$} &  & \  \ \rotatebox{45}{$\vdots$} &  &
\end{tabular} $$ 
\begin{center}
Pascal Triangle
\end{center}

\noindent One of the combinatorial interpretations of the terms of the sequence $D_d:={r+d\choose d}_{r=0}$, $d\in\N$, has to do with the counting of the lattice points associated with the dilations $r\Delta_d$ of the standard $d$-simplex $\Delta_d$. By the standard $d$-simplex $\Delta_d$ we mean the convex hull of the set $\{\underline{0}, e_1,\dots,e_d\}$ where $e_i's$, $1\leq i\leq d$ are the standard vectors in $\R^{d}$ and $\underline{0}$ is the origin. That is,
\begin{equation}
\Delta_{d}:= {\rm conv}(\underline{0},e_1,\dots,e_d)=\{{\bf x}\in\R ^{d} : {\bf x}\cdot e_{i}\geq 0, \ \ \sum_{i=1}^{d} {\bf x}\cdot e_{i}\leq 1\}
\end{equation}

\noindent and the dilation $r\Delta_{d}$, is  given by

\begin{equation}
r\Delta_{d}=\{{\bf x}\in\R ^{d} :{\bf x}\cdot e_{i}\geq 0, \ \ \sum_{i=1}^{d}{\bf x}\cdot e_{i}\leq r, \ \    r\in \N\}
\end{equation}

\noindent Lattice points are the points whose coordinates are integers. Asking for the lattice points  on $r\Delta_{d}$  is tantamount  to counting the integer solutions for the inequality
\begin{equation}
\sum_{i=1}^{d}{\bf x}\cdot e_{i}\leq r
\end{equation}

\noindent The number of lattice points on any given lattice polytope is well known. This is central theme of Ehrhart polynomials, [3], [6], [10], [11] and [16]. In fact the number of the lattice points on $r\Delta_d$ is given by 
\begin{equation}
|r\Delta_{d}\cap \Z_{\geq 0}^{d}|= {{r+d}\choose d}
\end{equation}

and its generating function by

\begin{equation}
{\rm P}(r\Delta_{d},z)=\sum_{r=0}^{\infty} A_{r}z^{r}=\frac{1}{(1-z)^{d+1}}, \ {\rm where} \ \ A_{r}={{r+d}\choose d}
\end{equation}

\noindent On the other hand, Grassmannians are ubiquitous in nature and they constitute one of the best understood  algebraic varieties. They admit algebraic, combinatorial and geometric structures which are very elegant. Their classical cohomology theory has taken the center stage in algebraic combinatorics in recent years, see  [4], [5], [7] , [8], [9], [10] and [12].  It turns out that the lattice points on $r\Delta_d$ encode some vital information about the indexing partitions of the Schubert varieties contained in the Grassmannian $Gr(d, d+r)$. This sheds more light on the cohomology ring of the Grassmannian. It is well known that the multiplicative generators of the cohomology of the Grassmannian Gr$(d,d+r)$ are given by the special Schubert cycles $\sigma_{\lambda}$, see [3]. These cycles are indexed by one-row partitions $\lambda=(k), 1\leq k\leq r$ and they constitute the total Chern class of the quotient bundle $\mathcal{Q}$, that is,
$$c(\mathcal{Q})= 1+\sigma_{\Yboxdim{7pt}\tiny\yng(1)} +\sigma_{\Yboxdim{7pt}\tiny\yng(2)}+ \cdots + \sigma_{{\Yboxdim{7pt}\tiny\yng(2)\cdots {\Yboxdim{7pt}\tiny\yng(1)}}_{1\times r}}$$

\noindent We study the monomials identified with the semi standard tableaux of these one-row Young diagrams  and realize a natural graded polynomial $T_{r}(t)$ called dilation polynomial. This is our first  refinement of the Ehrhart polynomial $L_{\Delta}(r)$ of the standard $d$-simplex $\Delta_{d}$. It comes with the natural weight $(1,1,\dots, 1)$. The second refinement is the the Poincar\'e polynomial {\bf P}$(Gr(d,d+r),z)$  of the Grassmannian Gr$(d,d+r)$ interpreted as the slicing of $r\Delta_{d}$ with hyperplanes with respect to the weight $(1,2,\dots,d)$.  It is interesting to note that the natural grading on the cohomology of the Grassmannian Gr$(d, d+r)$ induces  different gradings  of the lattice points of the dilation $r\Delta_{d}$ which give various refinements of the Ehrhart polynomial $L_{\Delta}(r)$. The paper is a generalisation of the previous studies in [1] and [2].  In section 2, we introduce a technique of counting lattice points by grading with respect to the weight ${\bf a}=(1,1,\dots, 1)$. 
 This is just the slicing of the dilation $r\Delta_d$ into parallel regular $(d-1)$-simplices. The
The polynomial
\begin{equation}
T^{(1,\ldots, 1)}_{r}(t) = \sum_{k=0}^r {k+d-1 \choose d-1}t^k 
\end{equation}
 refines the Ehrhart polynomial $L_{\Delta}(r)$. We give a generating function for such polynomials as $r$ grows. This grading allows us to establish in Section 3, a bijection between the lattice points of  the dilation $r\Delta_{d}$ and the semi standard tableaux of row Young diagrams indexing the special Schubert cycles of the Grassmanninan Gr$(d,d+r)$.  By using another weight ${\bf h}=(1,2, \ldots, d)$ which gives a different  slicing of the simplex, we construct a polynomial 
\begin{equation} P^{(1,2, \ldots, d)}_{r\Delta_d}(z) =\left[ k+d \choose d\right]_z \mbox{ for } 0\leq k \leq r
\label{eq:gauss}
\end{equation}
which is a $z$-binomial coefficient. This gives a bijection between the lattice points in $r\Delta_d$ and partitions fitting into an $r\times d$ rectangle, and establishes that the grading given here to a lattice point eventually identifies this polynomial with the Poincar\'e polynomial of the Grassmannian ${\rm Gr}(d, d+r)$. 

\section{ The  Dilation Polynomial $T_{r\Delta_{d}}, r\geq1$    }\label{S:flag}
\noindent We define the lexicographical order $<_{\rm lex}$ on the  set $r\Delta_{d}\cap \Z_{\geq 0}^{d}$ of lattice points on $r\Delta_{d}$ as follows: Let ${\bf a}=(a_1,\dots,a_d)$ and ${\bf b}=(b_1,\dots,b_d)$ be any two lattice points in $r\Delta_{d}\cap \Z_{\geq 0}^{d}$. We say ${\bf a} <_{\rm lex} {\bf b}$ if, in the integer coordinate difference  ${\bf a}-{\bf b}\in\Z^{d}$, the leftmost nonzero entry is negative. As noted earlier, the set $r\Delta_{d}\cap \Z_{\geq 0}^{d}$ of lattice points on $r\Delta_{d}$ is the integer solution set of the inequality $(1.3)$. It turns out that the upper bound $r$ in $(1.3)$ defines a relation on the lattice points of the solution set which brings about the disjoint subdivisions of the integer solution set.\

\begin{proposition}
 Let ${\bf a}$ and ${\bf b}$ be two lattice points in $r\Delta_{d}\cap \Z_{\geq 0}^{d}$ such that ${\bf a}<_{\rm lex} {\bf b}$.  The relation ${\bf a}\sim{\bf b}$ defined by $\sum^{d}_{i=1}(a_{i}-b_{i})=0$ is an equivalence relation.\\
\end{proposition}
\noindent The relation partitions  the set $r\Delta_{d}\cap \Z_{\geq 0}^{d}$ into disjoint equivalence classes. Notice that the integer solution set is complete with respect to the bound $r$ in the sense that the sum of integer coordinates of the lattice points in $r\Delta_{d}\cap \Z_{\geq 0}^{d}$ takes all the values of the integers in the closed interval $[0,r]$. Completeness is one of the beautiful properties of the standard d-simplex not all the lattice polytopes enjoy this feature.
\begin{corollary}
 Any two lattice points in $r\Delta_{d}\cap \Z_{\geq 0}^{d}$ belong to the same class if and only if they share the same sum of their respective integer coordinates.\\
\end{corollary}
\begin{corollary}
 $|r\Delta_{d}\cap \Z_{\geq 0}^{d}/\sim|=r+1$
 \end{corollary}
\begin{proof}
This  follows corollary 2  and  the  fact  that   r$\Delta_{d}\cap \Z_{\geq 0}^{d}$   is  complete 
 
\begin{equation}
r\Delta_{d}\cap \Z_{\geq 0}^{d}/\sim :=\{X_k:\sum^{d}_{i=1}x_i=k, \ 0\leq k\leq r, \forall x=(x_1,\dots,x_d)\in X_k\}
\end{equation} and hence, $|r\Delta_{d}\cap \Z_{\geq 0}^{d}/\sim|=r+1$

\end{proof}
\begin{corollary}
 The class of the origin $\underline{0}\in r\Delta_{d}\cap \Z_{\geq 0}^{d}$ is a singleton set.
 \end{corollary}
\begin{proof}
The class of the origin  denoted by $X_0$ is given by
\begin{equation}
X_0=\{\underline{x}=(x_1,\dots,x_d)\in r\Delta_{d}\cap \Z_{\geq 0}^{d} : \sum^{d}_{i=1}x_i=0\}
\end{equation}
Suppose that there is a lattice point {\bf a}   which belongs to $X_{0}$ such that {\bf  a} is not the origin. Since the origin $\underline{0}$ is $<_{\rm lex}$  than every lattice point ${\bf a}\in r\Delta_{d}\cap \Z_{\geq 0}^{d}$, so,  $\underline{0}\sim {\bf a}$ implies that $\sum^{d}_{i=1}(0-a_i)<0$, This integer value is not in $[0,r]$, therefore, there is no lattice point $r\Delta_{d}\cap \Z_{\geq 0}^{d}$ which is equivalent to the origin apart from itself hence $|X_0|=1$
\end{proof}

\noindent We now compute the size of each of the equivalence classes $X_k$ such that $0\leq k\leq r$.

\begin{theorem}
Let $\mathcal{A}= r\Delta_{d}\cap \Z_{\geq 0}^{d}$ denote the set of lattice points on  $r\Delta_d$ and let $X_k\subset \mathcal{A}$ be the collection of lattice points whose sum of their integer coordinates is $k$ such that $0\leq k\leq r$. Then $|X_k|={k+d-1\choose d-1}$.
\end{theorem}
\begin{proof}
Notice that the chain of the following inclusions
$$\{(0,\dots,0)\}\subset\Delta_{d}\cap \Z_{\geq 0}^{d}\subset 2 \Delta_{d}\cap \Z_{\geq 0}^{d} \cdots \subset r\Delta_{d}\cap \Z_{\geq 0}^{d}$$
implies the following chain.
$$\Delta_{d}\cap \Z_{> 0}^{d}\subset 2 \Delta_{d}\cap \Z_{> 0}^{d}\subset \cdots \subset r\Delta_{d}\cap \Z_{>0}^{d}$$
The subcollection $X_k$ is given by
$$X_k=\{\underline{x}=(x_1,\dots,x_d)\in \mathcal{A}: \ \ \ \sum_{i=1}^{d} x_i=k, \ \  0\leq k\leq r \}$$
$X_{0}=\{(0,\dots,0)\}$, so  $|X_k|=1$. Observe that 
$$X_k = k\Delta_{d}\cap \Z_{\geq 0}^{d}/(k-1)\Delta_{d}\cap \Z_{\geq 0}^{d},\ \ 2\leq k\leq r$$
In fact, $X_k's$  define a partition of the set  $r\Delta_{d}\cap \Z_{\geq 0}^{d}$ of the lattice points on $r\Delta_d$, that is,
$$\bigcap_{k=0}^{r}X_k=\emptyset, \ \ \ \ \ \ \bigcup_{k=0}^{r}X_k= \mathcal{A}$$
From Ehrhart theory,  using 1.4,
$$|\Delta_{d}\cap \Z_{\geq 0}^{d}|={1+d\choose d}=|X_0\cup X_1|.$$ This implies that $|X_1|=d$, Similarly, $$|2\Delta_{d}\cap \Z_{\geq 0}^{d}|={2+d\choose d}=|X_0\cup X_1\cup X_2|.$$ This gives
$$|X_2|={2+d\choose d}-d-1={1+d\choose d-1}$$
Continuing this way,\\
$$|X_k|={k+d\choose d}-\sum^{k}_{j=1}{k+d-j\choose d}={k+d-1\choose d-1}$$
\end{proof}

\noindent The disjoint union $\cup X_{k}$ of subcollections $X_k, \ \  0\leq k \leq r$  of the set 
$r\Delta_{d}\cap \Z_{\geq 0}^{d}$ of lattice points on $r\Delta_d$ defines  a polynomial $T_{r}(t)$ of degree $r$  in variable $t$ given by
\begin{equation}
T_{r}(t)=\sum_{k=0}^{r}{k+d-1\choose d-1}t^k
\end{equation}
\noindent We call $T_{r}(t)$  the dilation polynomial of degree $r$ identified with the dilation $r\Delta_d$. This is precisely the slicing  of  $r\Delta_d$ with hyperplanes perpendicular to the direction ${\bf a}:=(1,\dots,1)$
and enumerate all the  lattice points in the different layers. That is,
\begin{equation}
{k+d-1\choose d-1}=\#\{v\in r\Delta_{d}\cap \Z_{\geq 0}^{d} : v\cdot {\bf a}=k, 0\leq k\leq r \}
\end{equation}
\noindent The dilation polynomial $T_{4}(t)$ for the 4th dilation of the standard 3-simplex is illustrated in Figure 2. 

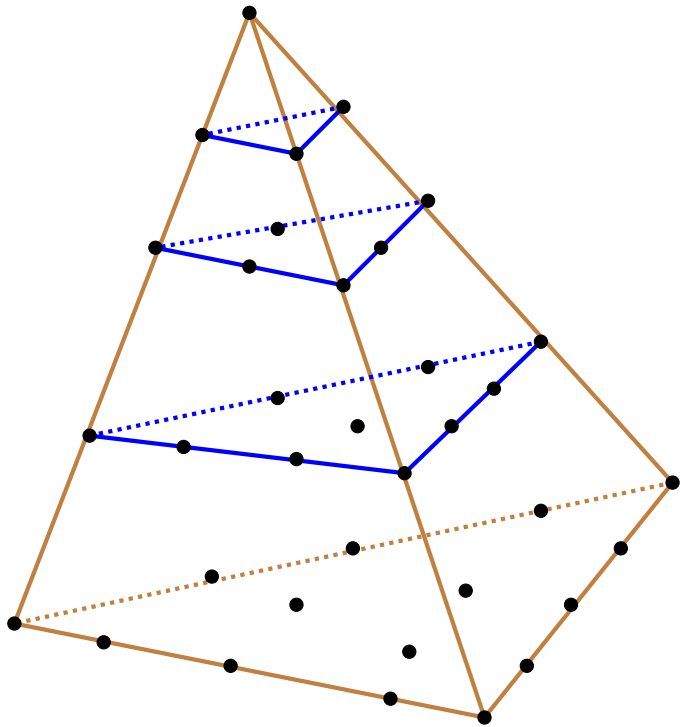
\begin{figure}[!h]
\centering
\begin{tikzpicture}[scale=1.25]
\draw[ultra thick,brown] (0,0)--(5,-1);
\draw[ultra thick,brown,dotted] (7,1.5)--(0,0);
\draw[ultra thick,brown] (2.5,6.5)--(5,-1);
\draw[ultra thick,brown] (2.5,6.5)--(7,1.5);
\draw[ultra thick,brown] (0,0)--(2.5,6.5);
\draw[ultra thick,brown] (5,-1)--(7,1.5);
\draw[ultra thick,blue] (0.8,2)--(4.15,1.6);
\draw[ultra thick,blue] (4.15,1.6)--(5.6,3);
\draw[ultra thick,blue,dotted] (5.6,3)--(0.8,2);
\draw[ultra thick,blue] (1.5,4)--(3.5,3.6);
\draw[ultra thick,blue] (3.5,3.6)--(4.4,4.5);
\draw[ultra thick,blue,dotted] (4.4,4.5)--(1.5,4);
\draw[ultra thick,blue] (2,5.2)--(3,5);
\draw[ultra thick,blue] (3,5)--(3.5,5.5);
\draw[ultra thick,blue,dotted] (3.5,5.5)--(2,5.2);
\filldraw[black](0,0)circle(2pt);
\filldraw[black](5,-1)circle(2pt);
\filldraw[black](7,1.5)circle(2pt);
\filldraw[black] (2.5,6.5)circle(2pt);
\filldraw[black](1.5,4)circle(2pt);
\filldraw[black](3.5,3.6)circle(2pt);
\filldraw[black](4.4,4.5)circle(2pt);
\filldraw[black] (0.8,2)circle(2pt);
\filldraw[black] (4.15,1.6)circle(2pt);
\filldraw[black] (5.6,3)circle(2pt);
\filldraw[black] (1.8,1.88)circle(2pt);
\filldraw[black] (3,1.75)circle(2pt);
\filldraw[black] (4.65,2.1)circle(2pt);
\filldraw[black] (5.1,2.5)circle(2pt);
\filldraw[black] (4.4,2.73)circle(2pt);
\filldraw[black] (2.8,2.4)circle(2pt);
\filldraw[black] (2.5,3.8)circle(2pt);
\filldraw[black] (3.9,4)circle(2pt);
\filldraw[black] (2.8,4.2)circle(2pt);
\filldraw[black] (0.95,-0.2)circle(2pt);
\filldraw[black] (2.3,-0.45)circle(2pt);
\filldraw[black] (4,-0.8)circle(2pt);
\filldraw[black] (6.45,0.8)circle(2pt);
\filldraw[black] (5.45,-0.45)circle(2pt);\filldraw[black] (5.92,0.2)circle(2pt);
\filldraw[black] (5.6,1.2)circle(2pt);
\filldraw[black] (3.6,0.8)circle(2pt);
\filldraw[black] (2.1,0.5)circle(2pt);
\filldraw[black] (3.65,2.1)circle(2pt);
\filldraw[black] (3,0.2)circle(2pt);
\filldraw[black] (4.8,0.35)circle(2pt);
\filldraw[black] (4.2,-0.3)circle(2pt);
\filldraw[black] (2,5.2)circle(2pt);
\filldraw[black] (3,5)circle(2pt);
\filldraw[black] (3.5,5.5)circle(2pt);
\end{tikzpicture}
\caption{$T_{4}(t)=1+3t+6t^{2}+10t^{3}+15t^{4}$}
\label{tetra1}
\end{figure}

\begin{remark}
Dilation polynomials identified with $r\Delta_2$ and $r\Delta_3$ are called triangular and tetrahedral polynomials respectively
\end{remark}
\begin{theorem}
Let $\mathcal{M}=\{T_{r}(t)\}_{r=0}$ be the sequence of dilation polynomials of lattice points counting on $r\Delta_d$ for $r\geq 0$. Then its generating series $G(t,z)=\sum_{r=0} T_{r}(t)z^r$ is given by
$$G(t,z)=\frac{z}{(1-z)(1-tz)^d}.$$
\end{theorem}

\begin{proof}
Notice from the equation (2.3) that 
$$T_{r}(t)=T_{r-1}(t)+\frac{(r+1)\cdots (r+d-1)}{(d-1)!}t^r \ \ {\rm and} \ \ \sum_{r\geq 0}\frac{(r+1)\cdots (r+d-1)}{(d-1)!}z^r = \frac{1}{(1-z)^d}.$$

$$G(t,z) =\sum_{r\geq 0} T_{r}(t)z^r=\sum_{r\geq 0}\begin{bmatrix} T_{r-1}(t)+\frac{(r+1)\cdots (r+d-2)}{(d-1)!}t^{r-1}\end{bmatrix}  z^r.$$
$$G(t,z)= zG(t,z)+ \sum_{r\geq 1}\begin{bmatrix}\frac{(r+1)\cdots (r+d-1)}{(d-1)!}t^{r-1}\end{bmatrix}z^r, \ \ {\rm and \ so}$$
$$G(t,z)= \frac{z}{(1-z)(1-tz)^d}.$$
\end{proof}

\section {The Cohomology ring of Grassmannian $Gr(d,d+r)$}

 \noindent Let $V$ be an $n$-dimensional  complex vector space. The set of all maximal chains of subspaces in $V$ is called the flag variety $\mathcal{F}\ell_{n}(\C)$ of dimension $\frac{n(n-1)}{2}$.  That is,
$$\mathcal{F}\ell_{n}(\C):=\{ V_{\bullet}:=\{0\}\subset V_{1}\subset V_{2}\subset\dots\subset V_{n}=V\ \ {\rm such\ \ that \ \ dim}V_{i}=i\}$$
\ \\
 \noindent The Grassmannian $Gr(d,n)$ is the spacial case of the flag variety being the set of all $d$-dimensional subspaces in $V.$ Its dimension is $d(n-d)$.  There is a projection  $$\pi : \mathcal{F}\ell _{n}(\C)\longrightarrow Gr(d,n)$$  from the full flag variety $\mathcal{F}\ell _{n}(\C)$ to the Grassmannian $Gr(d,n)$  with $\pi^{-1}(X_{\lambda}(F_{\bullet})) = X_{w(\lambda)}(F_{\bullet})$, where $X_{\lambda}(F_{\bullet})$ is a Schubert variety in the Grassmannian $Gr(d,n)$ defined as the closure of the Schubert cell $C_{\lambda}(F_{\bullet})$ given by 
 $$C_{\lambda}(F_{\bullet}) = \{V_{d}\in Gr(d,n) : {\rm dim}V_{d}\cap F_{n+i-\lambda_{i} }= i, \  1\leq i\leq d\},$$
 with respect to the fixed flag $F_{\bullet}$: 
 $$ F_{\bullet}:=\{0\}\subset F_{1}\subset F_{2}\subset\dots\subset F_{n}=V\ \ {\rm such\ \ that \ \ dim}F_{i}=i $$
 The partition $\lambda$ is called fitted in the sense that it has at most length $d$ and each part cannot exceed $n-d$. The permutation $w(\lambda)$  identified with the partition $\lambda=(\lambda_{1},\dots,\lambda_{d})$ is given by
 \begin{equation}
w_{i}=i+\lambda_{d+1-i}, \ 1\leq i\leq d  \ {\rm and} \ w_{j}< w_{j+1},\ d+1\leq j\leq n.
\end{equation}
This permutation  is called Grassmannian in that it has a unique descent by definition.  Every such permutation  has the code $c(w(\lambda))$ of the form $(w_{1}-1,w_{2}-2,\dots,w_{d}-d,0,\dots,0)$ which can be represented by $(m_1,m_2,\dots,m_d)$ by disregarding the string of zeros at the right hand. It turns out that the partition $\lambda$ indexing the Schubert variety $X_{\lambda}$ can be  recovered from  this code as $\lambda=(m_{i_1},m_{i_2},\dots,m_{i_d})$ where $m_{i_1}\geq m_{i_2}\geq\dots\geq m_{i_d}$ and $m_{i_p}\ne 0, \ 1\leq i_{p}\leq d$.  Recall that for any permutation $w$ in  the symmetric group $S_n$, the code $c(w)$ of $w$ is the sequence $(c_{1}(w),\dots,c_{n}(w))$ where $c_{i}(w)=\mid\{j: 1\leq i<j\leq n \ {\rm and} \ w(i)>w(j)\}\mid$. For instance the code $c(w)$ of the permutation $w=315426\in S_{6}$ is $(2,0,2,1,0,0)$. The string of zeros at the right hand may be discarded. Notice that $c_{i}(w)\leq n-i$. The length $\ell(w)$ of $w$ is $\#\{(i,j)\ : \ w(i)>w(j), 1\leq i<j\leq n\}$, the number of inversions in $w$, that is, the sum of integer coordinates of the code of $w$.  It is well known that the cohomology ring of the Grassmannian $Gr(d,n)$ is generated by the Schubert cycles $\sigma_{\lambda}$. These are Poincar\'e dual of the fundamental classes in the homology of Schubert varieties. The Grassmannian Gr$(d,n)$ admits many important vector bundles, most importantly there is a universal short exact sequence: $0\longrightarrow \mathcal{S} \longrightarrow \C^{n}\times {\rm Gr}(d,n)\longrightarrow \mathcal{Q}\longrightarrow 0$ of bundles on $Gr(d,n)$ which makes it easy to compute the Chern class $c(\mathcal{Q})$ of the quotient bundle $\mathcal{Q}$ on the Grassmannian $Gr(d,n)$. Recall that $\mathcal{Q}$ is a globally generated vector bundle of rank $r:=n-d$ and all its global sections are from the trivial bundle $\C^{d+r}\times {\rm Gr}(d,d+r)$. The total Chern class is the sum  over  all the one-row partitions inside the rectangle $\Box_{r\times d}$. That is,
\begin{equation}
c(\mathcal{Q})= 1+\sigma_{\Yboxdim{7pt}\tiny\yng(1)} +\sigma_{\Yboxdim{7pt}\tiny\yng(2)}+ \cdots + \sigma_{{\Yboxdim{7pt}\tiny\yng(2)\cdots {\Yboxdim{7pt}\tiny\yng(1)}}_{1\times r}}
\end{equation}

\noindent It turns out that the set  of all one-row Young diagrams indexing the multiplicative generators  of the cohomology of the Grassmannian Gr$(d,d+r)$ is deeply connected with the lattice points of $r\Delta_d$.  Let $\mathcal{C}_{d,r}$ be the set of row Young diagrams with at most $r$ boxes and adjoin the empty set $\phi$. That is,
$$\mathcal{C}_{d,r} = \{\Box_{1\times k} : \ 1\leq k\leq r \}\cup \emptyset.$$
The filiing of the boxes of the row Young diagrams in $\mathcal{C}_{d,r}$ using the numbers in $[d]:=\{1,\dots,d\}$ is semi standard, that is, the numbers weakly increase from the left to the right. We denote the collection of all such fillings by $\mathcal{C}_{d,r}^d$ and call it the d-filling set of the dilation $r\Delta_d$. For instance, the 3-filling set $\mathcal{C}_{3,3}^{3}$ associated the second dilation $3\Delta_3$ of the standard 3-simplex is the following collection
\ \\
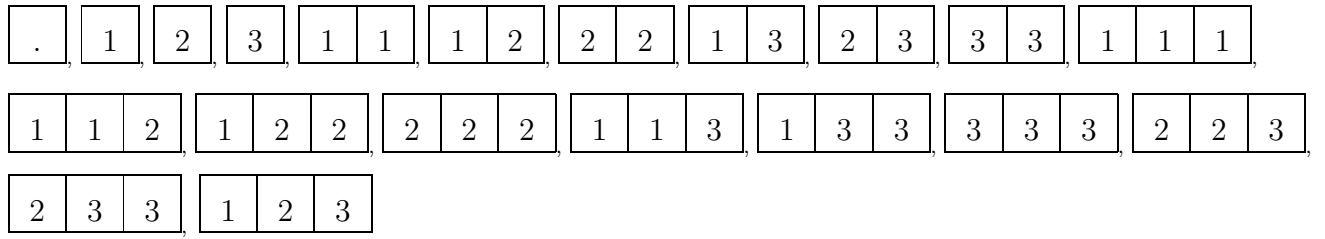
$$ \vcenter{\hbox{\Yboxdim{22pt}\young(.)}}_{,} \
 \vcenter{\hbox{\Yboxdim{22pt}\young(1)}}_{,} \  \vcenter{\hbox{\Yboxdim{22pt}\young(2)}}_{,} \  \vcenter{\hbox{\Yboxdim{22pt}\young(3)}}_{,} \
\vcenter{\hbox{\Yboxdim{22pt}\young(11)}}_{,} \  \vcenter{\hbox{\Yboxdim{22pt}\young(12)}}_{,} \ \vcenter{\hbox{\Yboxdim{22pt}\young(22)}}_{,} \  \vcenter{\hbox{\Yboxdim{22pt}\young(13)}}_{,} \  \vcenter{\hbox{\Yboxdim{22pt}\young(23)}}_{,} \ \vcenter{\hbox{\Yboxdim{22pt}\young(33)}}_{,}\  \vcenter{\hbox{\Yboxdim{22pt}\young(111)}}_{,}$$ 
$$ \vcenter{\hbox{\Yboxdim{22pt}\young(112)}}_{,}\  \vcenter{\hbox{\Yboxdim{22pt}\young(122)}}_{,}\  \vcenter{\hbox{\Yboxdim{22pt}\young(222)}}_{,}\  \vcenter{\hbox{\Yboxdim{22pt}\young(113)}}_{,}\  \vcenter{\hbox{\Yboxdim{22pt}\young(133)}}_{,}\  \vcenter{\hbox{\Yboxdim{22pt}\young(333)}}_{,} \  \vcenter{\hbox{\Yboxdim{22pt}\young(223)}}_{,}$$
$ \vcenter{\hbox{\Yboxdim{22pt}\young(233)}}_{,}\  \vcenter{\hbox{\Yboxdim{22pt}\young(123)}}$ \\
\  \\
\noindent These 20 semi standard Young tableaux can be organized  in terms of their defining  Young diagrams. It turns out that  this  arrangement  can be expressed as  a  polynomial, given by  $T_{3}(t)=1+3t +6t^2+10t^3$. This is the graded semi-standard  polynomial of degree $3$  illustrated in  Figure 2. \\

\begin{figure}[!hbt]
\begin{center}
\begin{tikzpicture}
\draw[fill] (0,0,0) circle [radius=0.075];
\draw[fill] (2,0,0) circle [radius=0.075];
\draw[fill] (0,2,0) circle [radius=0.075];
\draw[fill] (0,0,2) circle [radius=0.075];
\draw[ultra thick,red] (0,2,0)--(0,0,0);
\draw[ultra thick,red](0,2,0)--(0,0,2);
\draw[ultra thick,red](0,2,0)--(2,0,0);
\draw[ultra thick,brown] (2,0,0)--(0,0,0)--(0,0,2)--(2,0,0);
\node at (0,0,0){${\begin{tabular}{|c|}
\cline{1-1}  2   \\
\cline{1-1} \multicolumn{1}{c}{} \end{tabular}}$};
\node at (2,0,0){${\begin{tabular}{|c|}
\cline{1-1}  3  \\
\cline{1-1} \multicolumn{1}{c}{} \end{tabular}}$};
\node  at (0,2,0) {$  {\begin{tabular}{|c|}
\cline{1-1}  .   \\
\cline{1-1} \multicolumn{1}{c}{} \end{tabular}}$};
\node[left] at (0,0,2) {${\begin{tabular}{|c|}
\cline{1-1}  1  \\
\cline{1-1} \multicolumn{1}{c}{} \end{tabular}}$};					
\draw[fill] (0,-2,0) circle [radius=0.075];
\draw[fill] (3,-1,0) circle [radius=0.075];
\draw[fill] (1.5,-1.5,0) circle [radius=0.075];  
\draw[fill] (0,-2,4) circle [radius=0.075];					
\draw[fill] (0,-2,2) circle [radius=0.075];
\draw[fill] (1.5,-1.5,2) circle [radius=0.075];
\draw [ultra thick,red] (0,0,0)--(0,-2,0);
\draw [ultra thick,red] (2,0,0)--(3,-1,0);     									
\draw [ultra thick,red](0,0,2)--(0,-2,4);
\draw [ultra thick,green] (3,-1,0)--(1.5,-1.5,0)--(0,-2,0)--(0,-2, 2)--(0,-2,4)--(1.5,-1.5,2)--(3,-1,0);
\node[left] at (0,-2,4) {${\begin{tabular}{|c|c|}
\cline{1-2}  1 & 1   \\
\cline{1-2} \multicolumn{1}{c}{} \end{tabular}}$};
\node at (1.5,-1.5,0)  {${\begin{tabular}{|c|c|}
\cline{1-2}  2 & 3   \\
\cline{1-2} \multicolumn{1}{c}{} \end{tabular}}$};
\node  at (0,-2,2)  {${\begin{tabular}{|c|c|}
\cline{1-2}  1 & 2   \\
\cline{1-2} \multicolumn{1}{c}{} \end{tabular}}$};
\node[right] at (1.5,-1.5,2)  {${\begin{tabular}{|c|c|}
\cline{1-2}  1 & 3   \\
\cline{1-2} \multicolumn{1}{c}{} \end{tabular}}$};
\node at (0,-2,0)   {${\begin{tabular}{|c|c|}
\cline{1-2}  2 & 2  \\
\cline{1-2} \multicolumn{1}{c}{} \end{tabular}}$};
\node[right] at (3,-1,0)   {${\begin{tabular}{|c|c|}
\cline{1-2}  3 & 3  \\
\cline{1-2} \multicolumn{1}{c}{} \end{tabular}}$};
\draw [ultra thick,red] (0,-2,0)--(0,-4,0);
\draw [ultra thick,red] (3,-1,0)--(5,-3,0);     									
\draw [ultra thick,red] (0,-2,4)--(0,-5,7);
\draw[fill] (0,-5,7) circle [radius=0.075];
\draw[fill] (1.5,-4.4,4.9) circle [radius=0.075];
\draw[fill] (0,-4,0) circle [radius=0.075];  
\draw[fill] (2,-3.6,0) circle [radius=0.075];					
\draw[fill] (5,-3,0) circle [radius=0.075];
\draw[fill] (3.5,-3.3,0) circle [radius=0.075];
\draw[fill] (0,-4.3, 2.1) circle [radius=0.075];
\draw[fill] (0,-4.6,4.2) circle [radius=0.075];
\draw[fill] (3.5,-3.6,2.1) circle [radius=0.075];
\draw[fill] (1.45,-3.95,2.1) circle [radius=0.075];
\draw[ultra thick,blue](0,-5,7)--(0,-4.6,4.2)--(0, -4.3, 2.1)--(0,-4,0);
\draw[ultra thick,blue](0,-4,0)--(2,-3.6,0)--(3.5,-3.3,0)--(5,-3,0);
\draw[ultra thick,blue](5,-3,0)--(1.5,-4.4,4.9)--(3.5,-3.6,2.1)--(0,-5,7);
\node[left] at (0,-5,7) {${\begin{tabular}{|c|c|c|}
\cline{1-3}  1 & 1 & 1 \\
\cline{1-3} \multicolumn{1}{c}{} \end{tabular}}$};
\node[right] at (5,-3,0)  {${\begin{tabular}{|c|c|c|}
\cline{1-3}  3 & 3 & 3 \\
\cline{1-3} \multicolumn{1}{c}{} \end{tabular}}$};
\node at (0,-4,0)  {${\begin{tabular}{|c|c|c|}
\cline{1-3}  2 & 2 & 2 \\
\cline{1-3} \multicolumn{1}{c}{} \end{tabular}}$};
\node at (0,-4.6,4.2) {${\begin{tabular}{|c|c|c|}
\cline{1-3}  1 & 1 & 2 \\
\cline{1-3} \multicolumn{1}{c}{} \end{tabular}}$};
\node at (0,-4.3,2.1) {${\begin{tabular}{|c|c|c|}
\cline{1-3}  1 & 2 & 2 \\
\cline{1-3} \multicolumn{1}{c}{} \end{tabular}}$};
\node at (2,-3.6,0) {${\begin{tabular}{|c|c|c|}
\cline{1-3}  2 & 2 & 3 \\
\cline{1-3} \multicolumn{1}{c}{} \end{tabular}}$};
\node at (3.5,-3.3,0) {${\begin{tabular}{|c|c|c|}
\cline{1-3}  2 & 3 & 3 \\
\cline{1-3} \multicolumn{1}{c}{} \end{tabular}}$};
\node[right] at (1.5,-4.4,4.9) {${\begin{tabular}{|c|c|c|}
\cline{1-3}  1 & 1 & 3 \\
\cline{1-3} \multicolumn{1}{c}{} \end{tabular}}$};
\node[right] at (3.5,-3.6,2.1) {${\begin{tabular}{|c|c|c|}
\cline{1-3}  1 & 3 & 3 \\
\cline{1-3} \multicolumn{1}{c}{} \end{tabular}}$};
\node[right] at (1.45,-3.95,2.1) {${\begin{tabular}{|c|c|c|}
\cline{1-3}  1 & 2 & 3 \\
\cline{1-3} \multicolumn{1}{c}{} \end{tabular}}$};
\end{tikzpicture}
\caption{$T_{3}(t)=P_{3}(t)=1+3t+6t^{2}+10t^{3}$}
\end{center}
\end{figure}
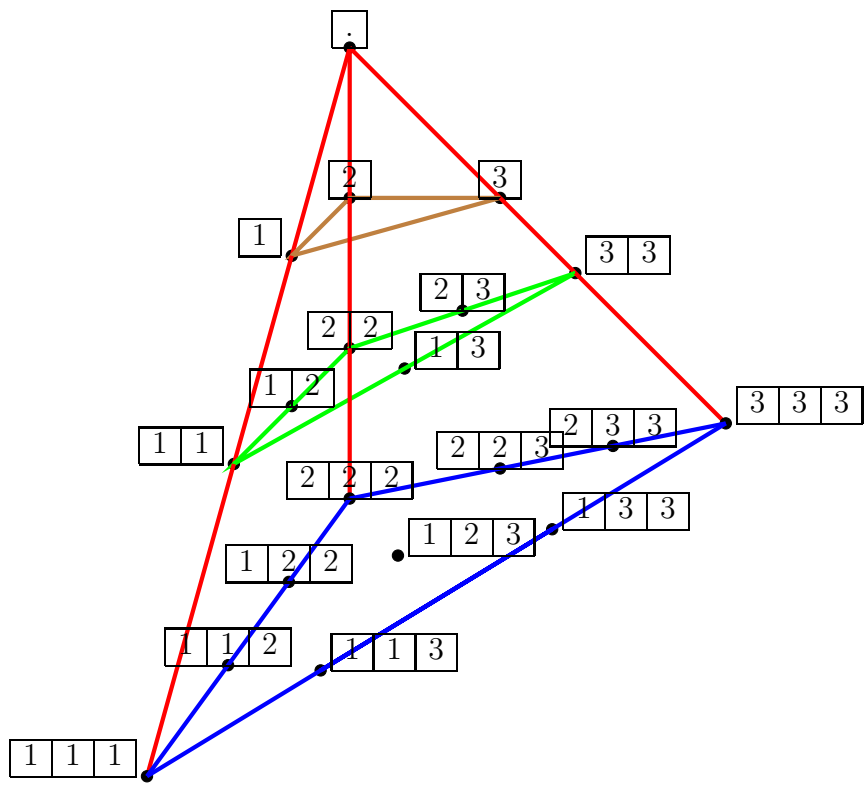

\begin{theorem} 
 \begin{enumerate}
 \item[(i.)] The size $L^{d}(r)$ of the d-filling set $\mathcal{C}_{d,r}^d$ is ${{r+d}\choose d}$ and the sequence $(L^{d}(r))_{r=0}^{\infty}$ of cardinalities  as r grows is recorded by the generating function
$$P(\mathcal{C}^{d}_{(d,r)},z)=\frac{1}{(1-z)^{d+1}}$$
\item[(ii.)]More is true, there is a graded counting polynomial of the semi standard tableaux  in $\mathcal{C}_{d,r}^d$   given by 
$$P_{r}(t)=\sum_{k=0}^{r} {k+d-1\choose d-1}t^{k}$$
that is, a  $k$-box row diagram gives ${k+d-1\choose d-1}$  semi standard Young tableaux. This has a generating function
$$G(t,z) = \frac{z}{(1-z)(1-tz)^d}.$$
\end{enumerate}
\end{theorem}

\begin{theorem}
There is a bijection $T\mapsto v(T)$ between the set $\mathcal{C}_{d,r}^{d}$  and the set $r\Delta_d\cap \Z_{\geq 0}^{d}$ of the lattice points of the  dilation $r\Delta_d$. Furthermore, the semi-standard polynomial $P_{r}(t)$ is precisely  the dilation polynomial $T_{r}(t)$  identified with $r\Delta_d$.
\end{theorem}
\begin{proof}
 To each semi standard tableau  $T\in \mathcal{C}_{(d,r)}^{d}$  there exists a unique exponent vector $v(T):=(v(T)_1,\dots, v(T)_{d})$ in which the coordinate $v(T)_{j}$  is the number of appearances of $j$ in $T$,\ $1\leq j\leq d$.Thi is a bijection.\\
 \ \\
  The number of semi standard fillings of each of the row diagram with shape $\lambda=(k), 0\leq k\leq r$ using the elements of the set $\{1,\dots,d\}$ has a well known closed formula.   Notice that for a fixed point ${\bf a} =(1,\dots,1)$ the following identity holds
 
$$\prod_{1\leq i < j\leq d} \frac{\lambda_{i}-\lambda_{j} +j-i}{j-i}={k+d-1\choose d-1}=\#\{v\in r\Delta_{d}\cap \Z_{\geq 0}^{d} : v\cdot {\bf a}=k, 0\leq k\leq r \}$$
Therefore, the semi-standard polynomial $P_{r}(t)$ can be viewed as the dilation polynomial $T_{r}(t)$. The bijection is a polynomial preserving map, see figure 2
\end{proof}

\section{Grassmannian Monomials}
\noindent It is clear from the Theorem 3.1 that every standard tableau $T\in \mathcal{C}_{(d,r)}^{d}$ defines a monomial ${\bf t}^{v(T)}$ where   $v(T):=(v(T)_1, \dots, v(T)_d)$, that is,  
\begin{equation}
{\bf t}^{v(T)}:=\prod_{j=1}^{d} t_{j}^{\# {\rm \ times \ j \ appears  \ in \   T}}, \ \  {\rm where}  \  \ v(T)\in r\Delta_{d}\cap\Z_{\geq 0}^{d}
\end{equation}
\noindent For instance, the monomial  defined by $T=\vcenter{\hbox{\Yboxdim{22pt}\young(11233)}}\in \mathcal{C}_{(4,5)}^{4}$ is given by  ${\bf t}^{\bf a}=t_{1}^{2}t_{2}t_{3}^{2}$ where ${\bf a}=(2,1,2,0)$. We call such monomials  in $\mathcal{C}_{(d,r)}^{d}$ Grassmannian  because they encode the data of indexing partitions of Schubert varieties in the Grassmannian Gr$(d,d+r)$. We denote these monomials by $W_{d}^{r}$, that is,
$$W_{d}^{r}:=\{t_1^{a_1}\cdots t_{d}^{a_d}  : \  \sum^{d}_{i=1} a_{i}\leq r, \  \  0\leq a_i\leq r\}$$

\begin{proposition}
Let $W_{d}^{r}$ and $W_{d}^{r'}$  be two Grassmannian monomial sets such that $r\leq r'$. Then $W_{d}^{r}\subseteq W_{d}^{r'}$.
\end{proposition}

\begin{proposition}
Every monomial ${\bf t}^{\bf a}\in \Z[t_1,\dots, t_d]$ is Grassmannian.
\end{proposition}
\begin{proof}
It suffices to produce a Grassmannian set  $W_{d}^{r}$  containing ${\bf t}^{\bf a}$. By (4.1) there is a semi-standard tableau $T$ which encodes the exponent vector ${\bf a}$ and this implies that there exists $r\in \N$ such that  $T$ is an element of  the $d$- filling set $\mathcal{C}_{(d,r)}^{d}$, so ${\bf t}^{\bf a}$ belongs to the Grassmannian monomial set $W_{d}^{r}$.
\end{proof}
\begin{corollary}
 If $r=\sum_{i=1}^{d} a_i$, where $a_{i}$ is an integer coordinate of {\bf a} then the Grassmannian set $W_{d}^{r}$ is the smallest set containing the monomial ${\bf t}^{\bf a}.$
\end{corollary}
\noindent It is important to quickly point out that the sum $P_{r}(t_1,\dots,t_d)$ of all the monomials in $W_{d}^{r}$, that is, 
\begin{equation}
P_{r}(t_1,\dots,t_d)=\sum_{T\in\mathcal{C}_{(d,r)}^{d}} \prod_{j=1}^{d} t_{j}^{\# {\rm \ times \ j \ appears  \ in \  T}}
\end{equation}
 is deeply connected with a polynomial representation $(V,\rho)$ of the general linear group $GL_{d}(\C)$  where  $V:= \bigoplus_{k=0}^{r}{\rm Sym}^{k}(\C^d)$ is the space of the direct sum of  homogeneous symmetric polynomials of dgree $k$ in $d$ variables. Let  $\C[X]:=\C[x_{11},x_{12}\dots,x_{dd}]$ be the ring of polynomial functions on $d\times d$ matrices. There is an action of $G=GL_{d}(\C)$ on $\C[X]$ by conjugation. The character of the polynomial representation $(V,\rho)$ is the  polynomial  $\chi_{\rho}\in\C[X]$ given by the trace of the matrix $\rho(X)$.  Recall that the character $\chi_{\rho}$ of every polynomial representation $(V, \rho)$ lies in the invariant ring $\C[X]^{G}$. Interested reader can consult  [15] and [17].
\begin{theorem}
The character $\chi_{V}$  of  $V:=\bigoplus_{k=0}^{r}{\rm  Sym}^{k}(\C^d)$  as a polynomial representation $\rho$ of the general linear group $GL_{d}(\C)$ is $P_{r}(t_1,\dots,t_d)$, that is,
$$\chi_{V}=\sum_{T\in\mathcal{C}_{(d,r)}^{d}} \prod_{j=1}^{d} t_{j}^{\# {\rm \ times \ j \ appears  \ in \  T}}$$
The sum ranging over all the semi standard fillings of the row diagrams with at most $r$ boxes.
\end{theorem}
\begin{proof}
Let  $t_1,\dots,t_d$ be eigenvalues of a generic $d\times d$ matrix $X$.  The map $\C[X]^{G} \longrightarrow \C[t_1,t_2,\dots,t_d] ^{\mathcal{S}_{n}}\ \ {\rm  defined \ by} f\mapsto f({\rm diag}(t_1,\dots,t_d))$  is an isomorphism. Set $\lambda =(k)$ since $k's$ define the rows diagrams with at most $r$ boxes, so the image of the character $f_{\rho}(X)$ is $$\sum_{k=0}^{r} \frac{{\rm det}(t_{i}^{\lambda_{i}+d-j})_{1\leq i,j\leq d}}{{\rm det}(t_{i}^{d-j})_{1\leq i,j\leq d}}.$$
\end{proof}
\begin{corollary}
 The dimension of the vector space $V:= \bigoplus_{k=0}^{r}{\rm Sym}^{k}(\C^d)$  is $\chi_{V}(1,1,\dots,1):=|r\Delta_d\cap \Z_{\geq 0}|,$ 
the number of  lattice points of the dilation $r\Delta_{d}.$
\end{corollary}
\begin{proof}
The Grassmannian set $W_{d}^{r}$  spans the vector space $V:= \bigoplus_{k=0}^{r}{\rm Sym}^{k}(\C^d).$
\end{proof}

\noindent Now to every monomial ${\bf t^a}\in \Z[t_1,\dots, t_d]$  we  associate a weight $w_{\bf a}$  defined by   
\begin{equation}
w_{\bf a}=\sum_{k=1}^{d} ka_k
\end{equation}
It turns out that $w_{\bf a}$ admits   two important partitions $\lambda,\lambda^{\ast}\vdash w_{\bf a}$ which can be identified with the monomial ${\bf t^a}$ .  These partitions, $\lambda$ and $\lambda^{\ast}$  are called  $\alpha$-partition and $\beta$-partition respectively. A partition $\lambda\vdash w_{\bf a}$  is said to be the $\alpha$-partition of the monomial  $t^{a_1}_1\cdots t_{d}^{a_d}$ if the number of parts of size $i$ in $\lambda$ is $a_{i}, \ 1\leq i\leq d$. The length $\ell(\lambda)$  of $\alpha$-partition is $a_1 +\cdots + a_d$. The $\beta$ partition $\lambda^{\ast}=(\lambda_{1}^{\ast},\dots,\lambda_{d}^{\ast})$ of $w_{\bf a}$  is such that $\lambda_{k}^{\ast}=\sum_{i\geq k}^{d} a_{i}, \ 1\leq k\leq d$ and its length is $d$. For instance, the $\alpha$-partition associated with the monomial $t_{1}^{3}t_{2}^{2}t_{3}^{3}t_{4}^{2}\in\Z[t_1,t_2,t_3,t_4]$ is (4,4,3,3,3,2,2,1,1,1) while its $\beta$ partition i $\lambda^{\ast}$ is $(10,7,5,2)$.  In fact,  $\alpha$ and $\beta$ partitions  identified with the monomial ${\bf t}^{\bf a}$ can be  realized in terms of the sum of the entries of the $d\times d$ upper triangular matrix  ${\rm M}_{\bf a}$ associated with  the exponent vector ${\rm a}=(a_1,\dots,a_d)$ of the monomial, that is,

\begin{equation}
{\rm M}_{\bf a}=\begin{bmatrix}
a_1 & a_2 & a_3 & \cdots & a_d\\
    & a_2 & a_3 & \cdots & a_d \\
    &     & a_3 & \cdots & a_d \\
    &     &     &        & \vdots\\
    &     &     &        & a_d
\end{bmatrix}
\end{equation}
The sum of the entries in the column $k$ divided by $k$ is the number of parts of size $k$ in the $\alpha$- partition $\lambda$ of $w_{\bf a}$. The $\beta$ partition $\lambda^{\ast}=(\lambda^{\ast}_1\dots \lambda^{\ast}_d)$ of $w_{\bf a}$  is such that $\lambda^{\ast}_k$ is the sum of the entries in the row $k$ where $1\leq k\leq d$. For instance, the $4\times 4$ matrix   ${\rm M}_{\bf a}$  corresponding to the monomial $t_{1}^{3}t_{2}^{2}t_{3}^{3}t_{4}^{2}\in\Z[t_1,t_2,t_3,t_4]$ is 
$${\rm M}_{\bf a}=\begin{bmatrix}
3 & 2 & 3 & 2\\
  & 2 & 3 & 2 \\
  &   & 3 & 2 \\
  &   &   & 2
\end{bmatrix}$$
\noindent so the $\alpha$-partition $\lambda$ and the $\beta$-partition  $\lambda^{\ast}$ identified with the matrix ${\rm M}_{\bf a}$ are  $1^{3}2^{2}3^{3}4^{2}$  and $(10,7,5,2)$ respectively.\\
\begin{proposition}
 Let $\lambda$ be the $\alpha$-partition  of $w_{\bf a}$  associated with the monomial ${\bf t}^{\bf a} = t_{1}^{a_1}\cdots t_{d}^{a_d}\in\Z[t_1,\dots,t_d]$. Then its corresponding $\beta$-partition $\lambda^{\ast}$ is the transpose of $\lambda$ and vice versa.
 \end{proposition}

\begin{proof} 
Let $\lambda=(\lambda_1,\dots,\lambda_{a_1+\cdots +a_d})$ and $\lambda^{\ast}=(\lambda_1^{\ast},\dots,\lambda_{d}^{\ast})$. It is obvious that these partitions satisfy the following identity
$$\sum_{k=1}^{a_1+\cdots +a_d}(2k-1)\lambda_k = \sum_{k=1}^{d}\lambda^{\ast 2}_{k}. $$
\end{proof}
\noindent It would be interesting to characterize and study all the monomials for which $\alpha$-partition and $\beta$-partition coincide. This amounts to the characterization of all self conjugate partitions. Recall that for all $n\in \mathbb{N}$ such that $n>2$ there is a bijection between the set of self conjugate partitions of $n$ and the set of all distinct odd parts partitions of $n$. For instance, a square free monomial of the form $t_{1}t_{2}\cdots t_{d}$ admits the stair case partition $(d, d-1,\dots, 1)$, this is deeply connected with the distribution of triangular numbers in the set $\N$ of natural numbers. We give a few other examples of such monomials.\\

\begin{example}
 Some monomials for which $\alpha$ and $\beta$-partitions coincide: 
\begin{enumerate}
\item[(i.)] All monomials of the form $t_{\frac{d}{2}}^{\frac{d}{2}}t_{d}^{\frac{d}{2}}\in\Z[t_1,t_2,\dots,t_d]$ for even $d$.\\
\item[(ii)] All monomials of the form  $t_{1}t^{d-2}t_d \in \Z[t_1,t_2,\dots,t_d]$.\\
\item[(iii.)] All monomials of the form  $t_1^{d-1}t_d\in \Z[t_1,t_2,\dots,t_d]$.\\
\item[(iv.)] All monomials of the form  $t_{d-2}t_{d-1}t_d^{d-2}\in \Z[t_1,t_2,\dots,t_d]$.
\end{enumerate}
\end{example}

\begin{lemma}
Let $\lambda^{\ast}=(\lambda_1,\dots,\lambda_{d})$ be the $\beta$-partition identified with the monomial $t^{a_1}_1\cdots t_{d}^{a_d}\in \Z[t_1,t_2,\dots,t_d]$. Then the exponent vector $(a_1,\dots,a_d)$ is equivalent to $(\lambda_1- \lambda_2, \lambda_2- \lambda_3,\dots,\lambda_{d-1}- \lambda_d,\lambda_{d}).$
\end{lemma}
\begin{proof}
It follows from the construction of the $\beta$ partition $\lambda^{\ast}$ from the exponent vector $(a_1,\dots,a_d)$. 
\end{proof}

\begin{theorem}
 Let ${\bf t^{\bf a}}\in W^{r}_{d}$ be a Grassmannian monomial associated with exponent vector ${\bf a}\in r\Delta_{d}\cap \Z_{\geq 0}^{d}$.  If a partition $\lambda^{\ast}$ is the $\beta$-partition identified with ${\bf t^{\bf a}}$  then the length $\ell(w(\lambda^{\ast}))$ of the Grassmannian permutation $w(\lambda^{\ast})$ is the weight $w_{\bf a}$.
 \end{theorem}
\begin{proof} The code $c(w(\lambda^{\ast}))$ of the Grassmannian permutation $w(\lambda^{\ast})$ is of the form\\  $(m_1,m_2,\dots,m_d,0,0,\dots,0)$. The rearrangement of $m_1,m_2,\dots,m_d$ in weakly decreasing order yields the fitted partition $\lambda^{\ast}=(\lambda^{\ast}_1,\dots,\lambda^{\ast}_d).$ The sum of the entries of the code $c(w)=(c_{1}(w),c_{2}(w),\dots,c_{n}(w))$ of any permutation $w$ is the length $\ell(w)$ of the partition, since each entry $c_{i}(w)$ is the number of inversions associated to the value $w_{i}$ in the position $i$. Hence the length $\ell(w(\lambda^{\ast}))$ of 
$w(\lambda^{\ast})$ is the size $|\lambda^{\ast}|$ of $\lambda^{\ast}$. Next we show that the weight $w_{\bf a}$ of the exponent vector ${\bf a}=(a_1,\dots,a_d)$ of the Grassmannian monomial ${\bf t^{\bf a}}=t^{a_1}_1\cdots t^{a_d}_d$ is $|\lambda^{\ast}|$. From Lemma 3.12 $a_i=\lambda^{\ast}_i-\lambda^{\ast}_{i+1},\ \ 1\leq i\leq d-1, \ \ a_d= \lambda^{\ast}_d$. Therefore, the weight $w_{\bf a}=\sum^{d-1}_{i=1}i(\lambda^{\ast}_i-\lambda^{\ast}_{i+1})+d\lambda^{\ast}_d=|\lambda^{\ast}|.$
\end{proof}
\begin{corollary}
Every $\beta$-partition $\lambda^{\ast}$ identified with each of the monomials ${\bf t}^{\bf a} \in W^{r}_d$ fits into the $r\times d$ rectangle $\Box_{r\times d}$. 
\end{corollary}
\begin{proof}
It is sufficient to establish that the parts of $\lambda^{\ast}$ cannot exceed $r$ and the length $\ell(\lambda^{\ast})$  of  $\lambda^{\ast}$ is $d$.  Notice that the exponent vector ${\bf a}$ is a lattice point of $r\Delta_{d}$ and by definition $a_1+\dots + a_d\leq r$. Therefore each part  $\lambda^{\ast}_{k}$ of $\lambda^{\ast}$ is at most $r$  and   length $\ell(\lambda^{\ast})$ is $d$ by the definition  of $\lambda^{\ast}$.  
\end{proof}
\begin{corollary}
The set of $\beta$-partitions $\lambda^{\ast}$ identified with  monomials in $W_{d}^{r}$ index the Schubert varieties in the Grassmannian Gr$(d,d+r)$, giving a bijection between lattice points in $r\Delta_{d}$ and partitions fitting into an $r\times d$ rectangle. 
\end{corollary}
\ \\
\noindent The weight $w_{\bf a}$ defined in the equation (3.2) gives another refinement $P^{h}_{r\Delta_d}(z)$ of the Ehrhart polynomial of $r\Delta_d$ with respect to a fixed point $h=(1,2,\dots d)$.  
 \begin{equation}
 P^{h}_{r\Delta_d}(z)=\sum_{m=0}^{dr} A_{m}z^m.
 \end{equation}
 \ \\
\noindent where $A_m = \#\{{\bf a}\in r\Delta_{d}\cap \Z_{\geq 0}^{d} : {\bf a}\cdot h = m, \ 0\leq m \leq dr \}$, that is, the number of exponent vectors ${\bf a}$  which share  the weight $m$.  We call  $P^{h}_{r\Delta_d}(z)$ the weighted polynomial associated with the dilation $r\Delta_{d}$.\\

\begin{lemma}
The polynomial $P^{h}_{r\Delta_d}(z)=\sum_{m=0}^{dr} A_{m}z^{m}$  specializes at $z=1$ to the Ehrhart polynomial $L_{\Delta_d}(r)$.
\end{lemma}

\begin{remark}
Notice that $A_{m}$ is precisely the number of lattice points in the intersection of the dilation  $r\Delta_d$ with the hyperplane $H_{m}$ perpendicular to the direction ${\bf h}:=(1,2,\dots,d)$. It is also interesting to note  that the grading given here to a lattice point eventually identifies the weighted  polynomial $P^{h}_{r\Delta_d}(z)$ with the Poincar\'e polynomial of the Grassmannian Gr$(d, d + r)$.
\end{remark}

\begin{theorem} Let $P^{h}_{r\Delta_2}(z)$ be the weighted polynomial of the lattice points of the dilation $r\Delta_{d}$ . Then the Poincar\'e polynomial $P(Gr(d,d+r),t)$  of the Grassmannian $Gr(d,d+r)$ coincides with the weighted polynomial $P^{h}_{r\Delta_d}(z).$
\end{theorem}
\begin{proof}
It is well known from the Borel presentation of the cohomology ring $H^{\ast}(Gr(d,d+r) ,\Z)$ of the Grassmannian Gr$(d, d+r)$ that the Poincar\'e polynomial P(Gr$(d,d+r), t$) is given by the following Gaussian polynomial
$$\frac{(1-t)(1-t{^2})\cdots (1-t^{d+r})}{(1-t)\cdots(1-t^{d})(1-t)\cdots (1-t^{r})}$$
This is combinatorially simplified as 
$$\sum_{\lambda\subseteq \Box_{d\times r} }t^{|\lambda|}$$
where $|\lambda|$ is the number of boxes in the Young diagram of shape $\lambda$. The size $|\lambda|$ coincides with the length $\ell(w(\lambda))$ (the number of inversions) of the Grassmannian permutation $w(\lambda)$ identified with $\lambda$ in the equation $(3.1)$. Notice that  $|\lambda|\leq dr$, therefore, It follows from the Theorem 4.9 that $|\lambda|$ is the weight $w_{\bf a}$ of the monomial $t^{\bf a}\in W_{d}^{r}$, $a\in r\Delta\cap\Z_{\geq 0}^{d}$, therefore, $\sum_{\lambda\subseteq \Box_{2\times r} }t^{|\lambda|}$ is precisely the polynomial  $\sum_{m=0}^{dr} A_{m}z^m$.
\end{proof}

\noindent {\it \bf Ouestion:} Does the set $r\Delta_{d}\cap \Z_{\geq 0}^{d}$ encode  some data about the degree and the Hilbert polynomial of Gr$(d, d+r)$?\\
 \ \\
\noindent {\bf Acknowledgment:} I would like to thank Dominic Bunett, Diane Maclagan and Mike Zabrocki  for  productive discussions during the preparation of the manuscript. I would also like to thank Balazs Szendr\"oi for his hospitality and contributions during my visit to the University of Oxford where the work was carried out. The author is supported by EPSRC GCRF grant EP/T001968/1, part of the Abram Gannibal Project.\\
\  \\

\begin{center}
{\bf References}
\end{center}

\begin{enumerate}
\item [1.]  P. Adeyemo, \emph{The lattice points of the dilations of the standard 2-simplex and the Grassmannian Gr(2,n) }, manuscript 2022.
\item [2.]  P. Adeyemo, \emph{The Lattice Points of the Standard 3-simplex and the Grassmannian Gr(3,n)}, manuscript 2022.
\item [3.] E. Ehrhart, \emph{Sur les polyedres rationnels homoth\'etiques a n dimensions,} C. R. Acad. Sci. Pari 254 (1962), 616-618.
\item [4.]  D. Eisenbud and J. Harris, \emph{3264 $\&$ All that intersection Theory in Algebraic Geometry}.
\item [5.] W. Fulton, \emph{Young tableaux}, volume 35 of London Mathematical Society Student Texts. Cambridge University Press, Cambridge, 1997.
\item [6.]  Gr\"{u}nbaum, B., Convex polytopes, volume 221 of Graduate Texts in Mathematics. Springer-Verlag, New York, 2003. ISBN 978-1-4613-0019-9
\item [7.] V. Laksmibai,  J. Brown,\emph{The Grassmannian variety: Geometric and Representation-Theoretic Aspects.} Developments in Mathematics. Vol. 42.  2015.
\item [8.] I. Macdonald, \emph{Symmetric Functions and Hall Polynomials},Oxford University Press, 2001.
\item [9.] A. Mendes and J. Remmel, \emph{Counting With Symmetric Functions}, Developments in Mathematics, Springer, pg. 292, 1st edition. 2015.
 \item [10.] E. Miller and B. Sturmfels, \emph{Combinatorial Commutative Algebra}, GTM 227, Springer-Verla, New York. 2000.
\item [11.] R. Simion, \emph{Convex Polytopes and Enumeration}, Advances in Applied Mathematics 18, 149-180 (1997).
\item [12.] F. Sottile, A. Morrison, \emph{Two Murnaghan-Nakayama rules in Schubert Calculus}. Annals of Combinatorics, 22(2), 363-375, 2018.
\item [13.] B. Sturmfels, \emph{Algorithms in Invariant Theory, 2nd Edition}, Texts $\&$ Monographs in Symbolic Computation. Springer Wien New York.1993.
\item [14.] B. Sturmfels, \emph{Gr\"{o}bner Bases and Convex Polytopes.} AMS University Lecture Series. Vol.8, American Mathematical Society, Providence, RI 1996.
\item [15.] B. Sturmfels and M. Michalek, \emph{Introduction to Nonlinear Algebra.} Graduate Studies in Mathematics. 211, American Mathematical Society, Providence, RI 2021.	
\end{enumerate}

\end{document}